\documentclass[11pt]{article}
\usepackage{amsmath,amsthm,amssymb,hyperref, color}
\usepackage{fullpage}

\newtheorem{thm}{Theorem}[section]
\newtheorem{claim}[thm]{Claim}
\newtheorem{lem}[thm]{Lemma}
\newtheorem{define}[thm]{Definition}
\newtheorem{cor}[thm]{Corollary}

\newtheorem{conjecture}[thm]{Conjecture}
\newtheorem{THM}{Theorem}

\newtheorem{rem}[thm]{Remark}

\def\L{{\mathcal L}}

\def\F{{\mathbb{F}}}

\def\R{{\mathbb{R}}}

\def\C{{\mathbb{C}}}

\def\cL{{\cal L}}

\def\m{{\mathbf{m}}}

\def\AP{{\mathbf{AP}}}

\newcommand{\ip}[2]{\langle #1,#2 \rangle}

\def\_{\,\,\,\,\,}

\def\supp{\textsf{supp}}

\def\rank{\textsf{rank}}

\newcommand{\eps}{\epsilon}

\newcommand{\remove}[1]{}
\newcommand{\floor}[1]{\left\lfloor #1\right\rfloor}

\begin{document}

\title{On the number of rich lines in truly high dimensional sets}

\author{Zeev Dvir\thanks{Department of Computer Science and Department of Mathematics,
Princeton University.
Email: \texttt{zeev.dvir@gmail.com}. } \and
Sivakanth Gopi\thanks{Department of Computer Science, Princeton University.
Email: \texttt{sgopi@cs.princeton.edu}.}}

\date{}
\maketitle

\begin{abstract}
We prove a new upper bound on the number of $r$-rich lines (lines with at least $r$ points) in a `truly' $d$-dimensional configuration of points $v_1,\ldots,v_n \in \C^d$. More formally, we show that, if the number of $r$-rich lines is significantly larger than $n^2/r^d$ then there must exist a large subset of the points  contained in a hyperplane. We conjecture that the factor $r^d$ can be replaced with a tight $r^{d+1}$. If true,  this would generalize the classic Szemer\'edi-Trotter theorem  which gives a bound of $n^2/r^3$ on the number of $r$-rich lines in a planar configuration. This conjecture was shown to hold in $\R^3$ in the seminal work of Guth and Katz \cite{GK10} and was also recently proved  over $\R^4$ (under some additional restrictions) \cite{SS14}. For the special case of arithmetic progressions ($r$ collinear points that are evenly distanced) we give a bound that is tight up to low order terms, showing that a $d$-dimensional grid achieves the largest number of $r$-term progressions.

The main ingredient in the proof is a new method to find a low degree polynomial that vanishes on many of the rich lines. Unlike previous applications of the polynomial method, we do not find this polynomial by interpolation. The starting observation is that the degree $r-2$ Veronese embedding takes $r$-collinear points to $r$ linearly dependent images. Hence, each collinear $r$-tuple of points, gives us a dependent $r$-tuple of images. We  then use the design-matrix method of \cite{BDWY12} to convert these `local' linear dependencies into a global one, showing that all the images lie in a  hyperplane. This then translates into a low degree polynomial vanishing on the original set.
\end{abstract} 

\section{Introduction}

The Szemer\'edi-Trotter theorem gives a tight upper bound on the number of incidences between a collection of points and lines in the real plane. We write $A\lesssim B$ to denote $A\le C\cdot B$ for some absolute constant $C$ and $A \approx B$ if we have both $A \lesssim B$ and $B \lesssim A$.
\begin{thm}[\cite{ST83}]
\label{thm-Szemeredi-Trotter}
Given a set of points $V$ and a set of lines $\L$ in $\R^2$, let $I(V,\L)$ be the set of incidences between $V$ and $\L$. Then,
$$I(V,\L)\lesssim |V|^{2/3}|\L|^{2/3}+|V|+|\L|.$$
\end{thm}  
This fundamental theorem has found many applications in various areas (see \cite{Dvir-survey} for some examples) and is  known to also hold in the 
complex plane $\C^2$ \cite{Toth,Zahl}.  In recent years there has been a growing interest in high dimensional variants of line-point incidence bounds \cite{SSZ13, Kollar14, Rudnev14, SS14, ST12,BS14}. This is largely due to the breakthrough results of Guth and Katz \cite{GK10} who proved the Erd\"os distinct distances conjecture. One of the main ingredients in their proof was an incidence theorem for configurations of lines in $\R^3$  satisfying some  `truly 3 dimensional' condition (e.g, not too many lines in a plane). The intuition is that, in high dimensions, it is `harder' to create many incidences between points and lines. This intuition is of course false if our configuration happens to lie in some low dimensional space. In this work we prove stronger line-point incidence bounds for sets of points that do not contain a large low dimensional subset. 

To state our main theorem we first  restate the  Szemer\'edi-Trotter theorem as a bound on the number of $r$-rich lines (lines containing at least $r$ points)  in a given set of points. Since our results will hold over the complex numbers we will switch now from $\R$ to $\C$. The complex version of Szemeredi-Trotter was first proved by Toth \cite{Toth} and then proved using different methods by Zahl \cite{Zahl}. For a finite set of points $V$, we denote by $\cL_r(V)$ the set of $r$-rich lines in $V$. The following is equivalent to Theorem-\ref{thm-Szemeredi-Trotter} (but stated over $\C$).

\begin{thm}[\cite{Toth,Zahl}]
\label{thm-r-richlines-2dimensions}
Given a set  $V$ of $n$ points in $\C^2$, for $r\ge 2$,
$$|\cL_r(V)|\lesssim \frac{n^2}{r^3}+\frac{n}{r}.$$
\end{thm}

Theorem~\ref{thm-r-richlines-2dimensions} is tight since a two dimensional square grid of $n$ points contains $\gtrsim n^2/r^3$ lines that are $r$-rich. We might then ask whether a $d$-dimensional grid $G_d = \{1,2,\ldots,h\}^d$, with $h \approx n^{1/d}$, has asymptotically the maximal number of $r$-rich lines among all $n$-point configurations that do not have a large low dimensional subset. It can be shown that for $r\ll_d n^{1/d}$,$$|\cL_r(G_d)| \approx_d \frac{n^2}{r^{d+1}},$$  where the subscript $d$ denotes that the constants in the inequalities may depend on the dimension $d$~\cite{SoVu04}.
Clearly, we can obtain a larger number of rich lines in $\C^d$ if $V$ is a union of several low dimensional grids. For example, for some $\alpha\gg_d 1$ (We use $A \gg B$ to mean $A \geq C\cdot B$ for some sufficiently large constant $C$) and $d>\ell>1$, we can  take a disjoint union of $r^{d-\ell}/\alpha$ $\ell$-dimensional grids $G_{\ell}$ of size $\alpha n/r^{d-\ell}$ each. Each of these grids will have $\gtrsim_d \alpha^2n^2/r^{2d-\ell+1}$ $r$-rich lines and so, together we will get $\gtrsim_d \alpha n^2/r^{d+1}$ rich lines. We can also take a union of $n/r$ lines containing $r$ points each, to get more $r$-rich lines than in the $d$-dimensional grid $G_d$ when $r\gg_d n^{1/d}$. We thus arrive at the following conjecture which, if true, would mean that the best one can do is to paste together a number of grids as above.

\begin{conjecture}\label{conj-rich}
For $r\ge 2$, suppose $V \subset \C^d$ is a set of $n$ points with $$|\cL_r(V)| \gg_d \frac{n^2}{r^{d+1}}+\frac{n}{r}.$$ Then there exists $1<\ell<d$ and  a subset $V' \subset V$ of size $ \gtrsim_d n/r^{d-\ell}$ which is contained in an $\ell$-flat (i.e., an $\ell$-dimensional affine subspace).
\end{conjecture}

This conjecture holds in $\R^3$ \cite{GK10} and, in a slightly weaker form, in $\R^4$ \cite{SS14}. We compare these two results with ours later in the introduction.
Our main result makes a step in the direction of this conjecture. First of all, our bound is off by a factor of $r$ from the optimal bound (i.e., with $n^2/r^{d}$ instead of $n^2/r^{d+1}$). Secondly, we are only able to detect a $d-1$ dimensional subset (instead of finding the correct $\ell$ which may be smaller).

\begin{THM}\label{thm-manyrichlines}
For all $d \geq 1$ there exists constants $C_d,C_d'$ such that the following holds. Let $V \subset \C^d$ be a set of $n$ points and let $r \geq 2$ be an integer. Suppose that for some $\alpha\ge 1$, $$|\cL_r(V)|\ge C_d\cdot\alpha\cdot \frac{n^2}{r^{d}}.$$  Then, there exists a subset $\tilde V \subset V$ of size at least $C_d'\cdot\alpha\cdot \frac{n}{r^{d-2}}$  contained in a $(d-1)$-flat. We can take the constants $C_d,C_d'$ to be $d^{cd},d^{c'd}$ for  absolute constants $c,c'>0$.
\end{THM}

Notice that the theorem is only meaningful when $r \gg d$ (otherwise the factor $r^d$ in the assumption will be swallowed by the constant $C_d$). On the other hand, if $r \gg n^{1/(d-1)}$ then the conclusion always holds. Hence, the theorem is meaningful when $r$ is in a `middle' range. Notice also that for $d=2,3$ and $r$ sufficiently small, the condition of the theorem also cannot hold, by the Szemeredi-Trotter theorem. However, when $d$ becomes larger, our theorem gives non trivial results (and becomes closer to optimal for large $d$). The proof of Theorem~\ref{thm-manyrichlines} actually  shows (Lemma~\ref{lem-findpoly})  that, under the same hypothesis, most of the rich lines must be contained in a  hypersurface of degree smaller than $r$. This in itself can be very useful, as we will see in the proof of Theorem~\ref{thm-newsumproduct} which uses this fact to prove certain sum-product estimates. The existence of such a low-degree hypersurface containing most of the curves can also be obtained when there are many $r$-rich curves of bounded degree with `two degrees of freedom' i.e. through every pair of points there are at most $O(1)$ curves (see Remark~\ref{rem-bounded-degree-curves}).

\paragraph*{Counting arithmetic progressions}
 An $r$-term arithmetic progression in $\C^d$ is simply a set of $r$ points of the form $\{y,y+x,y+2x, \ldots,y+(r-1)x\}$ with $x,y \in \C^d$. This is a special case of $r$ collinear points and, for this case, we can derive a tighter bound than for the general case. In a nutshell, we can show that a $d$-dimensional grid contains the largest number of $r$-term progressions, among all sets that do not contain a large $d-1$ dimensional subset. The main extra property of  arithmetic progressions we use in the proof is that they behave well under products. That is, if we take a Cartesian product of $V$ with itself, the number of progressions of length $r$ squares. 

For a finite set $V \subset \C^d$, let us denote the number of $r$-term arithmetic progressions contained in $V$ by $\AP_r(V)$. We first observe that, for all sufficiently small $r$, the grid $G_d$ (defined above) contains at least $\gtrsim_d n^2/r^d$ $r$-term progressions. To see where the extra factor of $r$ comes from, notice that the $2r$-rich lines in $G_d$ will contain $r$ arithmetic progressions of length $r$ each. Our main theorem shows that this is optimal, as long as there is no large low-dimensional set.
  
\begin{THM}\label{thm-apbound}
Let $0 < \epsilon < 1$ and $V \subset \C^d$ be a set of size $n$ and suppose that for some $r \geq 4$ we have $$ \AP_r(V) \gg_{d,\epsilon} \frac{n^2}{r^{d-\eps}}.$$ Then, there exists a subset $\tilde V \subset V$ of size $\gtrsim_{d,\epsilon} \frac{n}{r^{{2d/\eps}-1}}$  contained in a hyperplane.
\end{THM}


\subsection{Related Work}
Using the incidence bound between points and lines in $\R^3$ proved in~\cite{GK10}, one can prove the following theorem from which Conjecture~\ref{conj-rich} in $\R^3$ trivially follows (see Appendix~\ref{sec-stronger-thm1}).

\begin{thm}[\cite{GK10}]
\label{thm-conj-d=3}
Given a set $V$ of $n$ points in $\R^3$, let $s_2$ denote the maximum number of points of $V$ contained in a $2$-flat. Then for $r\ge 2$,
$$|\L_r(V)|\lesssim \frac{n^2}{r^4}+\frac{ns_2}{r^3}+\frac{n}{r}.$$
\end{thm}

Similarly, using the results of in~\cite{SS14}, we can prove the following theorem from which a slightly weaker version of Conjecture~\ref{conj-rich} in $\R^4$ trivially follows (see Appendix~\ref{sec-stronger-thm1}).
\begin{thm}[\cite{SS14}]
\label{thm-conj-d=4}
Given a set $V$ of $n$ points in $\R^4$, let $s_2$ denote the maximum number of points of $V$ contained in a 2-flat and $s_3$ denote the maximum number of points of $V$ contained in a quadric hypersurface or a hyperplane. Then there is an absolute constant $c>0$ such that for $r\ge 2$,
$$|\L_r(V)|\lesssim 	2^{c\sqrt{\log n}}\cdot \left(\frac{n^2}{r^5}+\frac{ns_3}{r^4}+\frac{ns_2}{r^3}+\frac{n}{r}\right).$$
\end{thm}
We are not aware of any examples where  points  arranged on a quadric hypersurface in $\R^4$ result in significantly more rich lines than in a four dimensional grid. It is, however, possible that one needs to weaken  Conjecture~\ref{conj-rich} so that for some $1<\ell<d$, an $\ell$-dimensional hypersurface of constant degree (possibly depending on $\ell$) contains $\gtrsim_d n/r^{d-\ell}$ points. 

To make the comparison with the above theorems easier, Theorem~\ref{thm-manyrichlines} can be stated equivalently as follows:
\begin{thm}[Equiv. to Theorem~\ref{thm-manyrichlines}]
\label{thm-manyrichlines-equiv}
Given a set $V$ of $n$ points in $\C^d$, let $s_{d-1}$ denote the maximum number of points of $V$ contained in a hyperplane. Then for $r\ge 2$,
$$|\L_r(V)| \lesssim_d \frac{n^2}{r^d} + \frac{ns_{d-1}}{r^2}.$$
\end{thm}

In~\cite{SoVu04}, it was shown that $|\L_r(V)|\lesssim_d \frac{n^2}{r^{d+1}}$ when $V\subset \R^d$ is a  {\em homogeneous} set.  This roughly means that the point set is a perturbation of the grid $G_d$. In \cite{LaSo07}, the result was extended for pseudolines and homogeneous sets in $\R^n$ where pseudolines are a generalization of lines which include constant degree irreducible algebraic curves. Adding the homogeneous condition on a set is a much stronger condition (for sufficiently small $r$) than requiring that no large subset belongs to a hyperplane (however, we cannot derive these results from ours since our dependence on $d$ is suboptimal).

\subsection{Overview of the proof}
The main tool used in the proof of Theorem~\ref{thm-manyrichlines} is a rank bound for design matrices. A {\em design matrix} is a matrix with entries in $\C$ and whose support (set of non-zero entries) forms a specific pattern. Namely, the supports of different columns have small intersections, the columns have large support and rows are sparse (see Definition~\ref{def-designmatrix}). Design matrices were introduced in ~\cite{BDWY12,DSW12} to study quantitative variants of the Sylvester-Gallai theorem. These works prove  certain lower bounds on the rank of such matrices, depending only on the combinatorial properties of their support  (see Section~\ref{sec-design-matrices}). Such rank bounds can be used to give {\em upper bounds} on the dimension of point configurations in which there are many `local' linear dependencies. This is done by using the local dependencies to construct rows of a design matrix $M$, showing that its rank is  high and then arguing that the dimension of the original set is small since it must lie in the kernel of $M$.

Suppose we have a configuration of points with many $r$-rich lines. Clearly, $r \geq 3$ collinear points are also linearly dependent. However, this conclusion does not use the fact that $r$ may be larger than 3. To use this information, we observe that a certain map, called the Veronese embedding, takes $r$-collinear points to $r$ linearly dependent points in a larger dimensional space (see Section~\ref{sec-veronese-embedding}). Thus we can create a design matrix using these linear dependencies similarly to the constructions of~\cite{BDWY11,DSW12} to get an upper bound on the dimension of the {\em image} of the original set, under the Veronese embedding. We use this upper bound to conclude that there is a polynomial of degree $r-2$ which contains all the points in our original configuration. We then proceed in a way similar to the proof of the Joints conjecture by Guth and Katz~\cite{GK10b} to conclude that there is a hyperplane which contains many points of the configuration (by finding a `flat' point of the surface).

\subsection{Application: Sum-product estimates}

Here, we show a simple application of our techniques to prove sum product estimates over $\C$. The estimates we will get can also be derived from the  Szemer\'edi-Trotter theorem in the complex plane (see Section~\ref{sec-newsumproduct-usingST}) and we include them only as an example of how to use a higher dimensional theorem in this setting. We hope that future progress on proving Conjecture~\ref{conj-rich} will result in progress on sum product problems.

We begin with some notations. For two sets $A,B \subset \C$ we denote by $A+B = \{a+b \,|\, a,b \in A\}$ the sum set of $A$ and $B$. For a set $A \subset \C$ and a complex number $t \in \C$ we denote by $tA = \{ ta \,|\, a \in A \}$ the dilate of $A$ by $t$. Hence we have that  $A+tA = \{ a + ta'\,|\, a,a' \in A\}$.

\begin{THM}\label{thm-newsumproduct}
Let $A \subset \C$ be a set of $N$ complex numbers and let $1\ll C \ll \sqrt{N}$. Define the set $$T_C = \left\{ t \in \C \,\left|\, |A + tA| \leq \frac{N^{1.5}}{C\sqrt{\log N}} \right. \right\}.$$ Then, $ |T_C| \lesssim \frac{N}{C^2}.$
\end{THM}

By taking $C$ to be a large constant, an immediate corollary is:
\begin{cor}
Let $A \subset \C$ be a finite set. Then $$ | A+ A\cdot A| = |\{ a + bc \,|\, a,b,c \in A\}| \gtrsim \frac{|A|^{1.5}}{\sqrt{\log |A|}}.$$
\end{cor}

\subsection{Organization}

In Section~\ref{sec-prelim} we give some preliminaries, including on design matrices and the Veronese embedding. In Section~\ref{sec-manyrichlines} we prove Theorem~\ref{thm-manyrichlines}. In Section~\ref{sec-apbound} we prove Theorem~\ref{thm-apbound}. In Section~\ref{sec-newsumproduct} we prove Theorem~\ref{thm-newsumproduct}. In Appendix~\ref{sec-stronger-thm1} we give a possible strengthening of Conjecture~\ref{conj-rich} along with the proofs of Theorem~\ref{thm-conj-d=3} and Theorem~\ref{thm-conj-d=4}.

\subsection{Acknowledgements}
We thank Ben Green and Noam Solomon for helpful comments.  Research supported by NSF grant CCF-1217416 and by the Sloan fellowship. Some of the work on the paper was carried out during the special semester on `Algebraic Techniques for Combinatorial and Computational Geometry', held at the Institute for Pure and Applied Mathematics (IPAM) during Spring 2014.

\section{Preliminaries}\label{sec-prelim}

We begin with some notations.  For a vector $v \in \C^n$ and a set $I \subset [n]$ we denote by $v_I \subset \C^I$ the restriction of $v$ to indices in $I$. We denote the {\em support} of a vector $v \in \C^d$ by $\supp(v) = \{ i \in [d]\,|\, v_i \neq 0 \}$ (this notation is extended to matrices as well). For a set of $n$ points $V \subset \C^d$ and an integer $\ell$, we denote by $V^\ell \subset \C^{d\ell}$ its $\ell$-fold Cartesian product i.e. $V^\ell=V\times V\times \cdots \times V\ (\ell\ \mathrm{times})$ where we naturally identify $\C^d\times \C^d\times \cdots \times \C^d\ (\ell\ \mathrm{times})$ with $\C^{d\ell}$.

\subsection{Design matrices}
\label{sec-design-matrices}

Design matrices, defined in \cite{BDWY12}, are matrices that satisfy a certain condition on their support. 

\begin{define}[Design matrix]\label{def-designmatrix}
Let $A$ be an $m \times n$ matrix over a field $\F$. Let $R_1,\ldots,R_m \in \F^n$ be  the rows of $A$ and let $C_1,\ldots,C_n \in \F^m$ be the columns of $A$. We say that $A$ is a {\em $(q,k,t)$-design matrix} if
\begin{enumerate}
\item For all $i \in [m]$, $|\supp(R_i)| \leq q$.
\item For all $j \in [n]$, $|\supp(C_j)| \geq k$.
\item For all $j_1 \neq j_2 \in [n]$, $|\supp(C_{j_1}) \cap \supp(C_{j_2}) | \leq t$.
\end{enumerate}
\end{define}

Surprisingly, one can derive a general bound on the rank of complex design matrices, despite having no information on the values present at the non zero positions of the matrix. The first bound of this form was given in \cite{BDWY12} which was improved in \cite{DSW12}.

\begin{thm}[\cite{DSW12}]\label{thm-rankbound}
Let $A$ be an $m \times n$ matrix with entries in $\C$. If $A$ is a $(q,k,t)$ design matrix then the following two bounds hold
\begin{eqnarray}
\rank(A) \geq n - \frac{ntq^2}{k}.\label{eq-rank1}\\
\rank(A) \geq n - \frac{mtq^2}{k^2}.\label{eq-rank2}
\end{eqnarray}

\end{thm}

\subsection{The Veronese embedding}
\label{sec-veronese-embedding}

We denote by $$\m(d,r) = \binom{d+r}{d}$$ the number of monomials of degree at most $r$ in $d$ variables. We will often use the lower bound $\m(d,r) \geq (r/d)^d$.
 The Veronese embedding $\phi_{d,r} : \C^d \mapsto \C^{\m(d,r)}$ sends a point $a = (a_1,\ldots,a_d) \in \C^d$ to the vector of evaluations of all monomials of degree at most $r$ at the point $a$. For example, the map $\phi_{2,2}$ sends $(a_1,a_2)$ to $(1,a_1,a_2,a_1^2,a_1a_2,a_2^2)$. 
 We can identify each point $w \in \C^{\m(d,r)}$ with a polynomial $f_w \in \C[x_1,\ldots,x_d]$ of degree at most $r$ in an obvious manner so that the value $f_w(a)$ at a point $a \in \C^d$ is given by the standard inner product $\ip{w}{\phi_{d,r}(a)}$. We will use the following two easy claims.

\begin{claim}\label{cla-veronese-vanish}
Let  $V \subset \C^d$ and let $U = \phi_{d,r}(V) \subset \C^{\m(d,r)}$. Then  $U$ is contained in a hyperplane iff there is a non-zero polynomial $f \in \C[x_1,\ldots,x_d]$ of degree at most $r$ that vanishes on all points of $V$.
\end{claim}
\begin{proof}
Each hyperplane in $\C^{\m(d,r)}$ is given as the set of points having inner product zero with some  $w \in \C^{\m(d,r)}$. If we take the corresponding polynomial $f_w \in \C[x_1,\ldots,x_d]$ we get that it vanishes on $V$ iff $\phi_{d,r}(V)$ is contained in the hyperplane defined by $w$. 
\end{proof}

\begin{claim}\label{cla-veronese-line}
Suppose the  $r+2$ points $v_1,\ldots,v_{r+2} \in \C^d$ are collinear and let $\phi = \phi_{d,r} : \C^d \mapsto \C^{\m(d,r)}$. Then, the points $\phi(v_1),\ldots,\phi(v_{r+2})$ are linearly dependent. Moreover, every $r+1$ of the points $\phi(v_1),\ldots,\phi(v_{r+2})$ are linearly independent.
\end{claim}
\begin{proof}
Denote $u_i = \phi(v_i)$ for $i=1\ldots r+2$. To show that the $u_i$'s are linearly dependent it is enough to show that, for any $w \in \C^{\m(d,r)}$, if all the $r+1$ inner products $\ip{w}{u_1}, \ldots, \ip{w}{u_{r+1}}$ are zero, then the inner product $\ip{w}{u_{r+2}}$ must also be zero. Suppose this is the case, and let $f_w \in \C[x_1,\ldots,x_d]$ be the polynomial of degree at most $r$ associated with the point $w$ so that $\ip{w}{u_i} = f_w(v_i)$ for all $1 \leq i \leq r+1$. Since the points $v_1,\ldots,v_{r+2}$ are on a single line $L \subset \C^d$, and since the polynomial $f_w$ vanishes on $r+1$ of them, we have that $f_w$ must vanish identically on the line $L$ and so $f_w(v_{r+2})= \ip{w}{u_{r+2}} =0$ as well. 

To show the `moreover' part, suppose in contradiction that $u_1,\ldots,u_r$ span $u_{r+1}$. We can find, by interpolation, a non zero polynomial $f \in \C[x_1,\ldots,x_d]$ of degree at most $r$ such that $f(v_1)= \ldots = f(v_r) = 0$ and $f(v_{r+1})=1$. More formally, we can translate the line containing the $r+1$ points to the $x_1$-axis and then interpolate a degree $r$ polynomial in $x_1$ with the required properties using the invertibility of the Vandermonde matrix.  Now, let $w \in \C^{\m(d,r)}$ be the point such that $f = f_w$. We  know that $\ip{w}{u_i}=0$ for $i=1\ldots r$ and thus, since $u_{r+1}$ is in the span of $u_1,\ldots,u_r$, we get that   $f(v_{r+1}) = \ip{w}{u_{r+1}}=0$ in contradiction. This completes the proof.
\end{proof}

\subsection{Polynomials vanishing on grids}

We recall the Schwartz-Zippel lemma. 

\begin{lem}[\cite{Schwartz80,Zippel79}]\label{lem-SZ}
Let $S \subset \F$ be a finite subset of an arbitrary field $\F$ and let $f \in \F[x_1,\ldots,x_d]$ be a non-zero polynomial of degree at most $r$. Then
$$   \lvert \{ (a_1,\ldots,a_d) \in S^d \subset \F^d \,\,|\,\, f(a_1,\ldots,a_d)=0 \}\rvert \leq r \cdot |S|^{d-1}. $$
\end{lem}

An easy corollary is the following claim about homogeneous polynomials.

\begin{lem}\label{lem-SZ-hom}
Let $S \subset \F$ be a finite subset of an arbitrary field $\F$ and let $f \in \F[x_1,\ldots,x_d]$ be a non-zero homogeneous polynomial of degree at most $r$. Then
$$   \lvert \{ (1,a_2,\ldots,a_d) \in \{1\} \times S^{d-1}  \,\,|\,\, f(1,a_2,\ldots,a_d)=0 \}\rvert \leq r \cdot |S|^{d-2}. $$
\end{lem}
\begin{proof}
Let $g(x_2,\ldots,x_d) = f(1,x_2,\ldots,x_d)$ be the polynomial one obtains from fixing $x_1=1$ in $f$. Then $g$ is a polynomial of degree at most $r$ in $d-1$ variables. If $g$ was the zero polynomial then $f$ would have been divisible by $1-x_1$ which is impossible for a homogeneous polynomial. Hence, we can use Lemma~\ref{lem-SZ} to bound the number of zeros of $g$ in the set $S^{d-1}$ by $r \cdot |S|^{d-2}$. This completes the proof.
\end{proof}
 
 Another useful claim says that if a degree one polynomial (i.e., the equation of a hyperplane) vanishes on a large subset of the product set $V^\ell$, then there is another degree one polynomial that vanishes on a large subset of $V$.
\begin{lem}\label{lem-hyperplaneproduct}
Let $V \subset \C^d$ be a set of $n$ points and let $V^\ell \subset \C^{d\ell}$ be its $\ell$-fold Cartesian product. Let $H \subset \C^{d\ell}$ be an affine hyperplane such that $|H \cap V^\ell| \geq \delta \cdot n^\ell$. Then, there exists an affine hyperplane $H' \subset \C^d$ such that $|H' \cap V| \geq \delta \cdot n$.
\end{lem}
\begin{proof}
Let $h \in \C^{d\ell}$ be the vector perpendicular to $H$ so that $x \in H$ iff $\ip{x}{h}=b$ for some $b \in \C$. Observing the product structure of $\C^{d\ell} = (\C^d)^\ell$ we can write $h = (h_1,\ldots,h_\ell)$ with each $h_i \in \C^d$. W.l.o.g suppose that $h_1 \neq 0$. For each $a = (a_2,\ldots,a_\ell) \in V^{\ell-1}$ let $V^\ell_a = V \times \{a_2\} \times \ldots \{a_\ell\}$. Since there are $n^{\ell-1}$ different choices for $a \in V^{\ell-1}$, and since 
$$ |V^\ell \cap H| = \sum_{a \in V^{\ell-1}}|V^\ell_a \cap H|, $$
 there must be some $a$ with $|V^\ell_a \cap H| \geq \delta \cdot n$. Let $H' \subset \C^d$ be the hyperplane  defined by the equation 
$$ x \in H' \,\,\text{iff}\,\, \ip{x}{h_1}+ \ip{a_2}{h_2} + \ldots + \ip{a_\ell}{h_\ell} = b. $$ Then, $H' \cap V$ is in one-to-one correspondence with the set $V^\ell_a \cap H$ and so has the same size.
 \end{proof}

\subsection{A graph refinement lemma}
We will need the following simple lemma, showing that any bipartite graph can be refined so that both vertex sets have high minimum degree  (relative the to the original edge density).
\begin{lem}\label{lem-refine}
Let $G=(A\sqcup B, E)$ be a bipartite graph with  $E \subset A \times B$ and edge set $E\ne \phi$. Then there exists non-empty sets $A'\subset A$ and $B'\subset B$ such that if we consider the induced subgraph $G'=(A'\sqcup B',E')$ then 
\begin{enumerate}
	\item The minimum degree in $A'$ is at least $\frac{|E|}{4|A|}$
	\item The minimum degree in $B'$ is at least $\frac{|E|}{4|B|}$
	\item $|E'|\ge |E|/2$.
\end{enumerate}
\end{lem}
\begin{proof}
We will construct $A'$ and $B'$ using an iterative procedure. Initially let $A'=A$ and $B'=B$. Let $G'=(A'\sqcup B',E')$ be the induced subgraph of $G$. If there is a vertex in $A'$ with degree (in the induced subgraph $G'$) less than $\frac{|E|}{4|A|}$, remove it from $A'$. If there is a vertex in $B'$ with degree (in the induced subgraph $G'$) less than $\frac{|E|}{4|B|}$, remove it from $B'$. At the end of this procedure, we are left with sets $A',B'$ with the required min-degrees.  We can count the number of edges lost as we remove vertices in the procedure. Whenever a vertex in $A'$ is removed we lose at most  $\frac{|E|}{4|A|}$ edges and whenever a vertex from $B'$ is removed we lose at most  $\frac{|E|}{4|B|}$ edges. So $$|E'|\ge |E|-|A|\frac{|E|}{4|A|}-|B|\frac{|E|}{4|B|}\ge |E|/2.$$
\end{proof}

\section{Proof of Theorem~\ref{thm-manyrichlines}}\label{sec-manyrichlines}

The main technical tool will be the following lemma, which shows that one can find a vanishing polynomial of low degree, assuming each point is in many rich lines. 

\begin{lem}\label{lem-findpoly}
For each $d \geq 1$ there is a constant $K_d \leq 32(2d)^d$ such that the following holds. Let $V \subset \C^d$ be a set of $n$ points and let $r\geq 4$ be an integer. Suppose that, through each point $v \in V$, there are at least $k$ $r$-rich lines where $$k \geq K_d\cdot \frac{ n}{r^{d-2}}.$$ Then, there exists a non-zero polynomial $f \in \C[x_1,\ldots,x_d]$ of degree at most $r-2$ such that $f(v)=0$ for all $v \in V$. 

If we have the stronger condition that the number of $r$-rich lines through each point of $V$ is between $k$ and $8k$  then we can get the same conclusion (vanishing $f$ of degree $r-2$) under the weaker inequality $$k \geq K_d\cdot \frac{n}{r^{d-1}}.$$
\end{lem}
\begin{proof}

 Let $V = \{v_1,\ldots,v_n\}$ and let $\phi = \phi_{d,r-2} : \C^d \mapsto \C^{\m(d,r-2)}$ be the Veronese embedding with degree bound $r-2$. Let us denote $U = \{u_1,\ldots,u_n\} \subset \C^{\m(d,r-2)}$ with $u_i = \phi(v_i)$ for all $i \in [n]$.

We will prove the lemma by showing that $U$ is contained in a hyperplane and then using  Claim~\ref{cla-veronese-vanish} to deduce the existence of the vanishing polynomial. Let $M$ be an $n \times \m(d,r-2)$ matrix whose $i$'th row is $u_i = \phi(v_i)$. To show that $U$ is contained in a hyperplane, it is enough to show that $\rank(M) < \m(d,r-2)$. This will imply that the columns of $M$ are linearly dependent, which means that all the rows lie in some hyperplane.

We will now construct a design matrix $A$ such that $A \cdot M=0$. Since $\rank(A) + \rank(M) \leq n$, we will be able to translate a lower bound on the rank of $A$ (which will be given by Theorem~\ref{thm-rankbound}) to the required upper bound on the rank of $M$.  Each row in  $A$ will correspond to some collinear $r$-tuple in $V$. We will construct $A$ in several stages. First, for each $r$-rich line $L \in \cL_r(V)$ we will construct a set of $r$-tuples $R_L \subset \binom{V}{r}$ such that
\begin{enumerate}
	\item Each $r$-tuple in $R_L$ is contained in  $L \cap V$.
	\item Each point $v \in L \cap V$ is in at least one $r$-tuple from $R_L$.
	\item Every pair of distinct points $u,v \in L \cap V$ appear together in at most two $r$-tuples from $R_L$.
\end{enumerate}
If $|L \cap V|$ is a multiple of $r$, we can construct such a set $R_L$ easily by taking a disjoint cover of $r$-tuples. If $|L \cap V|$ is not a multiple of $r$ (but is still of size at least $r$) we can take a maximal set of disjoint $r$-tuples inside it and then add to it one more $r$-tuple that will cover the remaining elements and will otherwise intersect only one other $r$-tuple. This will guarantee  that the third condition holds. We define $R \subset {V \choose r}$ to be the union of all sets $R_L$ over all $r$-rich lines $L$.  We can now prove:

\begin{claim}\label{cla-rtuples}
The set $R \subset {V \choose r}$ defined above has the following three properties.
\begin{enumerate}
	\item Each point $v \in V$ is contained in at least $k$ $r$-tuples from $R$.
	\item Every pair of distinct points $u,v \in V$ is contained together in at most two $r$-tuples from $R$.
	\item Let $(v_{i_1},\ldots,v_{i_r}) \in R$. Then there exists $r$ non zero coefficients $\alpha_1,\ldots,\alpha_{r} \in \C$ so that $ \alpha_1 \cdot u_{i_1} + \ldots + \alpha_r \cdot u_{i_r} = 0$.
 \end{enumerate}
If, in addition, we know that each point belongs to at most $8k$ rich lines (as in the second part of the lemma) then we also have that  $|R|\le 16nk/r$.
\end{claim}
\begin{proof}
The first property follows from the fact that each $v$ is in at least $k$ $r$-rich lines and that each $R_L$ with $v \in L$ has at least  one $r$-tuple containing $v$. The second property follows from the fact that each pair $u,v$ can belong to at most one $r$-rich line $L$ and that each $R_L$ can contain at most two $r$-tuples with both $u$ and $v$. The fact that the $r$-tuple of point $u_{i_1},\ldots,u_{i_r}$ is linearly dependent follows from  Claim~\ref{cla-veronese-line}. The fact that all the coefficients $\alpha_j$ are non zero holds since no proper subset of that $r$-tuple is linearly  dependent (by the `moreover' part of Claim~\ref{cla-veronese-line}). If each point is in at most $8k$ lines then each point is in at most $16k$ $r$-tuples (at most two on each line). This means  that there could be at most $16nk/r$ tuples in $R$ since otherwise, some point would be in too many tuples.
\end{proof}
 
We now construct the matrix $A$ of size $m\times n$ where $m=|R|$. For each $r$-tuple $(v_{i_1},\ldots,v_{i_r}) \in R$ we add a row to $A$  (the order of the rows does not matter) that has zeros in all positions except $i_1,\ldots,i_r$ and has values $\alpha_1,\ldots, \alpha_r$ given by Claim~\ref{cla-rtuples} in those positions. Since the rows of $M$ are the points $u_1,\ldots,u_n$, the third item of Claim~\ref{cla-rtuples} guarantees that $A \cdot M = 0$ as we wanted. The next claim asserts that $A$ is a design matrix.

\begin{claim}\label{cla-Adesign}
The matrix $A$ constructed above is a $(r,k,2)$-design matrix. 
\end{claim}
\begin{proof}
Clearly, each row of $A$ contains at most $r$ non zero coordinates. Since each point $v \in V$ is in at least $k$ $r$-tuples from $R$ we have that each column of $A$ contains at least $k$ non-zero coordinates. The size of the intersection of the supports of two distinct columns in $A$ is at most two by item (2) of Claim~\ref{cla-rtuples}.
\end{proof}

We now use Eq.~(\ref{eq-rank1}) from  Theorem~\ref{thm-rankbound} to get $$\rank(A)\ge n-\frac{2nr^2}{k}.$$ This implies (using $r \geq 4$) that
$$ \rank(M) \leq \frac{2nr^2}{k}  \leq \left(\frac{r-2}{d}\right)^d < \m(d,r-2),$$ if $$k \geq  2(2d)^d \cdot \frac{n}{r^{d-2}}.$$

If we have the additional assumption that each point is in at most $8k$ lines then, using the bound $m= |R| \leq 16nk/r $ in Eq.~(\ref{eq-rank2}) of Theorem~\ref{thm-rankbound}. We get

$$\rank(A) \geq n-\frac{2mr^2}{k^2} \geq n - \frac{32nr}{k}$$ which gives
$$ \rank(M) \leq \frac{32nr}{k}  < \m(d,r-2)$$  for 
$$ k  \geq 32(2d)^d \frac{n}{r^{d-1}}.$$ Hence, the rows of $M$ lie in a hyperplane. This  completes the proof of the lemma.
\end{proof}

\begin{rem}
\label{rem-bounded-degree-curves}
Lemma~\ref{lem-findpoly} can be extended to the case where we have $r$-rich curves of bounded degree $D=O(1)$ with `two degrees of freedom' i.e. through every pair of points there can be at most $C=O(1)$ distinct curves (e.g. unit circles). Under the Veronese embedding $\phi_{d,\floor{\frac{r-2}{D}}}$, the images of $r$ points on a degree $D$ curve are linearly dependent.  So we can still construct a design matrix as in the above proof where the design parameters depend on $D,C$. Once we get a hypersurface of degree $\floor{\frac{r-2}{D}}$ vanishing on all the points, the hypersurface should also contain all the degree $D$ $r$-rich curves.
\end{rem}

We will now use Lemma~\ref{lem-findpoly} to prove Theorem~\ref{thm-manyrichlines}. The reduction uses Lemma~\ref{lem-refine} to reduce to the case where each point has many rich lines through it. Once we find a vanishing low degree polynomial we analyze its singularities to find a point such that all lines though it are in some hyperplane.

\begin{proof}[Proof of Theorem~\ref{thm-manyrichlines}]
 Since $\cL_r(V)\le n^2$ for all $r\ge 2$, by choosing $C_d> R_d^d$ we can assume that $r\ge R_d$ for any large constant $R_d$ depending only on $d$.

Let $\cL = \cL_r(V)$ be the set of $r$-rich lines in $V$ and let $I=I(\L,V)$ 
be the set of incidences between $\L$ and $V$. By the conditions of the theorem we have
\begin{equation}\label{eq-largeI}
	|I| \ge r|\cL| \geq C_d\cdot\frac{\alpha n^2}{r^{d-1}}.
\end{equation}

Applying Lemma~\ref{lem-refine} to the incidence graph between $V$ and $\L$, we obtain non-empty subsets $V' \subset V$ and $\cL' \subset \cL$ such 
that each $v \in V'$ is in at least $k=\frac{|I|}{4n}$ lines from $\cL'$ and such that each line in 
$\cL'$ is $r/4$-rich w.r.t to the set $V'$ and $$|I'|=|I(\L',V')|\ge |I|/2.$$ We would like to  apply Lemma~
\ref{lem-findpoly} with the stronger condition that each point is incident on approximately the same number of lines (which gives better dependence on $r$). To achieve this, we will further refine our set of points using dyadic pigeonholing. 

Let $V'=V'_1\sqcup V'_2\sqcup \cdots$ be a partition of $V'$ into disjoint subsets where $V'_j$ is the set of points incident to at least $k_j=2^{j-1}k$ and less than $2^{j}k$ lines from $\L'$. Let $I'_j=I(\L',V'_j)$, so that $$\sum_{j\ge 1} |I_j'|=|I'|\ge |I|/2.$$ Since $\sum_{j\ge 1}\frac{1}{2j^2}<1$, there exists $j$ such that $|I'_j|\ge \frac{|I|}{4j^2}.$ Let us fix $j$ to this value for the rest of the proof.

We will first upper bound $j$. Since $|I_j'|>0$, $V_j'$ is non-empty and let $p\in V'_j$. There are at least $k_j$ $(r/4)$-rich lines through $p$ and by choosing $R_d\ge 8$, there are at least $r/4-1\ge r/8$ points other than $p$ on each of these lines and they are all distinct. So, 
$$n=|V|\ge 2^{j-1}k\cdot \frac{r}{8} = \frac{2^{j-6}r|I|}{n}\ge C_d \frac{2^{j-6}\alpha n}{r^{d-2}}\ge \frac{2^{j-6}n}{r^{d-2}}.$$ This implies
$j\lesssim  d\log r$ where we assumed above that $C_d\ge 1$.

Since the lines in $\L'$ need not be $r/4$-rich w.r.t $V'_j$, we need further refinement. Apply Lemma~\ref{lem-refine} again on the incidence graph $I'_j= I(\L',V'_j)$ to get non-empty $V''\subset V'_j$ and $\L''\subset \L'$ and $$|I''|=|I(\L'',V'')|\ge \frac{|I'_j|}{2}\ge \frac{|I|}{8j^2} \ge \frac{r|\L|}{8j^2}.$$ Each line in $\L''$ is incident to at least $$\frac{|I'_j|}{4|\L'|} \geq \frac{r}{16j^2}=r_0$$ points from $V''$ and so $\L''$ is $r_0$-rich w.r.t $V''$. And each point in $V''$ is incident to at least $$\frac{|I'_j|}{4|V'_j|}\ge \frac{k_j}{4}=2^{j-3} k=k_0$$  and at most $2^{j}k=8k_0$ lines from $\L''$. Since $j\lesssim d\log r$, we can assume $r_0=\frac{r}{16j^2}\ge 4$ by choosing $R_d \gg d^3$.

The following claim shows that we can apply Lemma~\ref{lem-findpoly} to $V''$ and $\L''$
\begin{claim} $k_0\ge K_d\cdot\frac{|V''|}{r_0^{d-1}}$ where $K_d$ is the constant in Lemma~\ref{lem-findpoly}
\end{claim}
\begin{proof}
We have
$$|V''|\le |V'_j| \le \frac{|I|}{2^{j-1}k}=\frac{n}{2^{j-3}}.$$  So it is enough to show that 
$$k_0\ge K_d\cdot\frac{ n}{2^{j-3} r_0^{d-1}}.$$
Substituting the bounds we have  for $k_0$ and $r_0$, this will follow from to $$|I| \ge 16K_d\cdot 2^{4d}\cdot \left(\frac{j^{2(d-1)}}{2^{2j}}\right) \frac{n^2}{r^{d-1}}$$ which follows from Eq.~(\ref{eq-largeI}) by choosing $C_d>16K_d\cdot 2^{4d}\cdot \max_j\left(\frac{j^{2(d-1)}}{2^{2j}}\right)$.
\end{proof} 
Hence, by Lemma~\ref{lem-findpoly}, there exists a non-zero polynomial $f \in \C[x_1,\ldots,x_d]$ of degree at most $r_0-2$, vanishing at all points of $V''$. W.l.o.g suppose $f$ has minimal total degree among all polynomials vanishing on $V''$. Since $f$ has degree at most $r_0-2$ it must vanish identically on all lines in $\cL''$.

We say that a point $v \in V''$ is `flat' if the set  of lines from $\cL''$ passing through $v$ are contained in some affine hyperplane through $v$. Otherwise, we call the point $v$ a `joint'. We will show that there is at least one flat point in $V''$. Suppose towards a contradiction that all points in $V''$ are joints. Let $v \in V''$ be some point and let $\nabla f(v)$ be the gradient of $f$ at $v$. Since $f$ vanishes identically on all lines in $\cL''$ we get that $\nabla f(v)=0$ ($v$ is a singular point of the hypersurface defined by $f$). We now get a contradiction since one of the coordinates of $\nabla f$ is a non-zero polynomial of degree smaller than the degree of $f$ that vanishes on the entire set $V''$. 

Hence, there exists a point $v \in V''$ and an affine a hyperplane $H$ passing through $v$ such that all $r_0$-rich lines in $\cL''$ passing through $v$ are contained in $H$. Since there are at least $k_0$ such lines, and each line contain at least $r_0-1$ points in addition to $v$, we get that $H$ contains at least $$(r_0-1)k_0 \geq \frac{r}{32j^2}\cdot 2^{j-3}\frac{|I|}{4n} \ge C_d \left(\frac{2^{j-10}}{j^2}\right) \frac{\alpha n}{r^{d-2}}\ge C_d' \frac{\alpha n}{r^{d-2}}$$ points from $V$ where $C_d'= C_d \cdot \min_j \left(\frac{2^{j-10}}{j^2}\right)$. Observing the proof we can take the constants to be $C_d=d^{\Theta(d)}$ and $C'_d=\frac{C_d}{2^{11}}$. 
\end{proof}

\begin{rem}\label{rem-hyperplane_lines}
Observe that, we can take $\L$ to be any subset of $\L_r(V)$ of size $\ge C_d \frac{\alpha n^2}{r^d}$ and obtain the same conclusion. Moreover, the hyperplane $H$ that we obtain at the end contains $k_0\gtrsim \frac{\alpha n}{r^{d}}$ lines of $\L$.
\end{rem}
\section{Proof of Theorem~\ref{thm-apbound}}\label{sec-apbound}
We will reduce the problem of bounding $r$-term arithmetic progressions to that of bounding $r$-rich lines using the following claim:
\begin{claim}\label{cla-ap_to_richline}
Let $V\subset \C^d$ then $\AP_r(V)\le |\L_r([r] \times V)|$ where $[r]=\{0,1,\cdots,r-1\}$
\end{claim}
\begin{proof}
For $u,w\in \C^d, w\ne 0$, let $(u,u+w,\cdots,u+(r-1)w)$ be an $r$-term arithmetic progression in $V$. Then the line $\{(0,u)+z(1,w)\}_{z\in \C}$ is $r$-rich w.r.t the point set $[r]\times V \subset \C^{1+d}$; moreover this mapping is injective.
\end{proof}
We need the following claim regarding arithmetic progressions in product sets.
\begin{claim}\label{cla-aptensor}
Let $V \subset \C^d$ be a set of $n$ points and let $\ell \geq 1$ be an integer. Then, for all $r \geq 1$, the product set $V^\ell \subset \C^{d\ell}$ satisfies $$ \AP_r(V^\ell) \geq \AP_r(V)^\ell.$$
\end{claim}
\begin{proof}
Let $P(V)$ be the set of $r$-term arithmetic progressions in $V$ and let $P(V^\ell)$ be the set of $r$-term progressions in $V^\ell$. We will describe an injective mapping from $P(V)^\ell$ into $P(V^\ell)$. For $u,w \in \C^d$ let $L_{u,w} = \{u,u+w,\ldots,u+(r-1)w\}$ be the $r$-term progression starting at $u$ with difference $w$. Let $u_1,\ldots,u_\ell,w_1,\ldots,w_\ell \in \C^d$ such that $L_{u_i,w_i} \in P(V)$ for each $i \in [\ell]$. We map them into the arithmetic progression $L_{u,w} \in P(V^\ell)$ with $u = (u_1,\ldots,u_\ell)$ and $w = (w_1,\ldots,w_\ell)$. Clearly, this map is injective (care should be taken to assign each progression a unique difference since these are determined up to a sign).
\end{proof}

\begin{proof}[Proof of Theorem~\ref{thm-apbound}]

Let us assume $\AP_r(V)\gg_{d,\epsilon} \frac{n^2}{r^{d-\epsilon}}$. Let $\ell=\lceil\frac{1}{\epsilon}\rceil$. By Claim~\ref{cla-aptensor}, $\AP_r(V^\ell)\ge \AP_r(V)^\ell$. Let $\L$ be the collection of $r$-rich lines w.r.t $[r]\times V^\ell\subset \C^{1+d\ell}$ corresponding to non-trivial $r$-term arithmetic progressions in $V^\ell$, as given by Claim~\ref{cla-ap_to_richline}. So
 $$|\L_r([r]\times V^\ell)|\ge |\L| = \AP_r(V^\ell)\ge \AP_r(V)^\ell\gg_{d,\epsilon} \frac{n^{2\ell}}{r^{d\ell-\epsilon\ell}}\ge \frac{n^{2\ell}}{r^{d\ell -1}}= \frac{(n^{\ell}r)^2}{r^{d\ell+1}}.$$ By Theorem~\ref{thm-manyrichlines} (choosing the constants appropriately), there is a hyperplane $H$ in $\mathbb{C}^{1+d\ell}$ which contains $\gtrsim_{d,\epsilon} \frac{n^\ell r}{r^{d\ell-1}}$ points of $[r]\times V^\ell$. Moreover, by Remark~\ref{rem-hyperplane_lines}, $H$ contains some of the lines of $\L$. So $H$ cannot be one of the hyperplanes $\{z_{1}=i\}_{i\in [r]}$ because they do not contain any lines of $\L$. So the intersection of $H$ with one of the $r$ hyperplanes $\{z_{1}=i\}_{i\in [r]}$ (say $j$) gives a $(d\ell-1)$-flat which contains  $\gtrsim_{d,\epsilon} \frac{n^\ell}{r^{d\ell-1}}$ points of $V^\ell\times \{j\}$. This gives a hyperplane $H'$ in $\C^{d\ell}$ which contains $\gtrsim_{d,\epsilon} \frac{n^\ell}{r^{d\ell-1}}$ points of $V^\ell$. Now by Lemma~\ref{lem-hyperplaneproduct}, we can conclude that there is a hyperplane in $\mathbb{C}^d$ which contains $\gtrsim_{d,\epsilon} \frac{n}{r^{d\ell-1}}\ge \frac{n}{r^{2d/\epsilon-1}}$ points of $V$.
\end{proof}
\section{Proof of Theorem~\ref{thm-newsumproduct}}\label{sec-newsumproduct}

Suppose in contradiction that $|T_C| > \lambda N/C^2$ for some large absolute constant $\lambda$ which we will choose later. Let $Q \subset T_C$ be a set of size  $$ |Q| = \left\lceil \frac{\lambda N}{C^2} \right\rceil$$ 
containing the zero element $0 \in Q$ (we have $0 \in T_C$ since the sum-set  $|A+0A| = |A|$ is small). Let us denote by $$r = |Q|$$ and let $$ m = \frac{N^{1.5}}{C\sqrt{\log N}}.$$
Let   $$	d = \lceil 100\log N \rceil.$$ We will  use our assumption on the size of $Q$ to construct a configuration of points $V \subset \C^d$ with many $r$-rich lines. Then we will use Lemma~\ref{lem-findpoly} to derive a contradiction. The set $V$ will be a union of the sets
$$	V_t = \{t\} \times (A+ tA)^{d-1} 	   = \{ (t,a_2 + t b_2, \ldots, a_d + tb_d) \,|\, a_i,b_j \in A \}$$ over all $t \in Q$. That is $$ V = \bigcup_{t \in Q} V_t.$$ Notice the special structure of the set 
$$ V_0 = \{0\} \times A^{d-1}. $$ We denote by 
\begin{equation}\label{eq-Vsize}
	 n = |V| \leq r \cdot m^{d-1}
\end{equation}

Notice that, by construction, for every $a = (0,a_2,\ldots,a_d)$ and every $b = (1,b_2,\ldots,b_d)$ (with all the $a_i,b_j$ in $A$), the line through the point $a \in V_0$ in direction $b$ is $r$-rich w.r.t  $V$. Let us denote by $\cL \subset \cL_r(V)$ the set of all lines of this form. We thus have
\begin{equation}
	|\cL|  = N^{2d-2}.
\end{equation}
Let $I=I(V,\L)$, then $|I|\ge r|\L|$.
We now use Lemma~\ref{lem-refine} to find subsets $V' \subset V$ and $\cL' \subset \cL$ such that each point in $V'$ is in at least 
$$ k = \frac{rN^{2d-2}}{4n}$$ lines from $\cL'$, each line in $\cL'$ is $r_0=r/4$-rich w.r.t to the set $V'$ and $$|I(V',L')|\ge |I|/2.$$ Observe that, since each line in $\L'$ contains at most $r$ points from $V'$, we have  $$|\L'|\ge |I(V',\L')|/r\ge |\L|/2.$$ The following claim shows that we can apply Lemma~\ref{lem-findpoly} on the set $V'$.
\begin{claim}
$$ k \geq K_d\frac{n}{r_0^{d-2}}.$$ where $K_d=32(2d)^d$ is the constant in Lemma~\ref{lem-findpoly}
\end{claim}
\begin{proof}
Plugging in the value of $k,r_0$ and rearranging, we need to show that
$$  \frac{N^{2d-2} r^{d-1}}{32(8d)^d} \geq n^2.$$
Using Eq.~(\ref{eq-Vsize}) to bound $n$ we get that it is enough to show
$$ \frac{N^{2d-2}r^{d-1}}{32(8d)^d} \geq  \frac{r^2 N^{3d-3}}{C^{2d-2} (\log N)^{d-1}}.$$
Rearranging, we need to show that
$$  r^{d-3} \geq \frac{32(8d)^{d}N^{d-1}}{(C^2)^{d-1}(\log N)^{d-1}}.$$
We now raise both sides to the power $1/(d-3)$ and use the fact that, for $\ell > \log X$, we have $1 \leq X^{1/\ell} \leq 2$. Thus it is enough to show
$$  r \geq \frac{K' d N}{C^2 \log N}$$ for some absolute constant $K'$.
Plugging in the value of $d$ we get that the claim would follow if
$$   r \geq \frac{100K'N}{C^2} $$ which holds by choosing $\lambda=100K'$.
\end{proof}

Since $C\ll\sqrt{N}$, $r_0\ge 4$. Applying Lemma~\ref{lem-findpoly}, we get a non-zero polynomial $f \in \C[x_1,\ldots,x_d]$ of degree at most $r_0-2$ that vanishes on all points in $V'$. This means that $f$ must also vanish identically on all lines in $\cL'$ (since these are all $r_0$-rich w.r.t $V'$). Since each line in $\cL'$ intersects $V_0$ exactly once, and since $|V_0| = N^{d-1}$, we get that there must be at least one point $v \in V_0$ that is contained in at least $|\cL'|/N^{d-1} \geq \frac{1}{2}N^{d-1}$ lines (in different directions) from $\cL'$. Let $\tilde f$ denote the homogeneous part of $f$ of highest degree. If $f$ vanishes identically on a line in direction $b \in \C^d$, this implies that $\tilde f(b) = 0$ (to see this notice that the leading coefficient of $g(t) = f(a + tb)$ is $\tilde f(b)$). Hence, since all the directions of lines in $\cL'$ are from the set $\{1\} \times A^{d-1}$, we get that $\tilde f$ has at least $\frac{1}{2}N^{d-1}$ zeros in the set $\{1\} \times A^{d-1}$. This contradicts Lemma~\ref{lem-SZ-hom} since the degree of $\tilde f$ is at most $r_0-2=r/4-2 < N/2$ (since $r = \lceil \lambda N/C^2 \rceil$ and $C \gg1$). This completes the proof of Theorem~\ref{thm-newsumproduct}.
\qed
\subsection{A proof of Theorem~\ref{thm-newsumproduct} using Szemer\'edi-Trotter in $\C^2$}
\label{sec-newsumproduct-usingST}
The following is a slightly stronger version of Theorem~\ref{thm-newsumproduct} (without the logarithmic factor), which we prove using a simple application of the two-dimensional Szemer\'edi-Trotter theorem (to derive Theorem~\ref{thm-newsumproduct} replace $C$ with $C\sqrt{N}$).
\begin{thm}
Let $A\subset \C$ be a set of $N$ complex numbers and let $C\gg1$. Define $$T_C=\left\{t\in \C: |A+tA|\le \frac{N^2}{C}\right\}.$$ Then 
$$|T_C|\lesssim \frac{N^2}{C^2}.$$
\end{thm}
\begin{proof}
Define the set of points $$P=\bigcup_{t\in T_C}\{t\}\times(A+tA)$$ and the set of lines $$\cL=\{(z,a+za')_{z\in \C}: a,a'\in A\}$$ in $\C^2$. Each line in $\cL$ contains $r=|T_C|$ points from $P$. So by using Theorem~\ref{thm-r-richlines-2dimensions}, we have
$$|\cL|\le K\left(\frac{|P|^2}{r^3}+\frac{|P|}{r}\right)$$ where $K$ is some absolute constant. By construction, $|P|\le |T_C|N^2/C=rN^2/C$ and $|\cL|=N^2$. So
$$\frac{N^2}{K}\le \frac{N^4}{rC^2}+\frac{N^2}{C}\Rightarrow r\le \frac{2KN^2}{C^2}$$
where we assumed that $C\ge 2K$.
\end{proof}

\bibliographystyle{alpha}
\bibliography{richlines}

\begin{appendix}

\section{Towards an optimal incidence theorem for points and lines in $\C^d$}
\label{sec-stronger-thm1}

For $\alpha \ge 1$ and some $1<\ell<d$, by pasting together $r^{d-\ell}/\alpha$ $\ell$-dimensional grids of size $\alpha n/r^{d-\ell}$ each, we get $\gtrsim_d \alpha n^2/r^{d+1}$ $r$-rich lines. This motivates a stronger version of Conjecture~\ref{conj-rich}.

\begin{conjecture}\label{conj-stronger}
Suppose $V \subset \C^d$ is a set of $n$ points and let $\alpha\ge 1$. For $r\ge 2$, if $$|\cL_r(V)| \gg_d \alpha\frac{n^2}{r^{d+1}}+\frac{n}{r},$$ then there is an integer $\ell$ such that $1<\ell<d$ and a subset $V' \subset V$ of size $ \gtrsim_d \alpha n/r^{d-\ell}$ which is contained in an $\ell$-flat.
\end{conjecture}

The above conjecture is equivalent to the following conjecture.
\begin{conjecture}[Equiv. to Conjecture~\ref{conj-stronger}]
\label{conj-equivalent-formulation}
Given a set $V$ of $n$ points in $\C^d$, let $s_\ell$ denote the maximum number of points of $V$ contained in an $\ell$-flat. Then for $r\ge 2$,
$$|\L_r(V)|\lesssim_d \frac{n^2}{r^{d+1}}+n\sum_{\ell=2}^{d-1}\frac{s_\ell}{r^{\ell+1}}+\frac{n}{r}.$$
\end{conjecture}
\begin{proof}[Proof of equivalence to Conjecture~\ref{conj-stronger}]
\ref{conj-equivalent-formulation} $\Rightarrow$ \ref{conj-stronger} is trivial. To show the other direction, if
$$|\cL_r(V)| \lesssim_d  \frac{n^2}{r^{d+1}}+\frac{n}{r},$$ we are done. Else, let 
$$|\cL_r(V)| = C_d\left( \alpha\frac{n^2}{r^{d+1}}+\frac{n}{r}\right)$$ for some $\alpha\ge 1$ and some $C_d\gg_d 1$. By \ref{conj-stronger}, for some $1<\ell<d$, we have $s_\ell \gtrsim_d \alpha n/r^{d-\ell}$. So we have  
$$|\cL_r(V)| \lesssim_d \left( \frac{ns_\ell}{r^{\ell+1}}+\frac{n}{r}\right).$$ This implies \ref{conj-equivalent-formulation}.
\end{proof}
Note that Theorem~\ref{thm-manyrichlines-equiv} and thus Theorem~\ref{thm-manyrichlines}, trivially follow from Conjecture~\ref{conj-equivalent-formulation} by observing that $s_{d-1}\ge s_{d-2}\ge \cdots \ge s_1 \ge r$. Conjecture~\ref{conj-equivalent-formulation}, if true, can be used to give an optimal bound on incidences between points and lines in $\C^d$ in terms of $s_\ell$'s by standard arguments.

For $d=2$, Conjecture~\ref{conj-equivalent-formulation} is exactly Theorem~\ref{thm-r-richlines-2dimensions}. Using the incidence bounds of \cite{GK10} and \cite{SS14} we can prove Conjecture~\ref{conj-equivalent-formulation} for $\R^3$ and `almost' prove it for $\R^4$ (these are stated as Theorem~\ref{thm-conj-d=3} and Theorem~\ref{thm-conj-d=4} respectively). As already discussed, it is possible that one needs to weaken  Conjecture~\ref{conj-equivalent-formulation} so that $s_\ell$ is the maximum number of points in an $\ell$-dimensional hypersurface of constant degree (possibly depending on $\ell$). 

For completeness, we include in this section a short derivation of Theorems~\ref{thm-conj-d=3} and \ref{thm-conj-d=4} from the incidence bounds of \cite{GK10} and \cite{SS14} which we now state.

\begin{thm}[\cite{GK10}]
\label{thm-incidencesinR3}
Let $V$ be a set of $n$ points and $\L$ be a set of $m$ lines in $\R^3$. Let $q_2$ be the maximum number of lines in a hyperplane (2-flat). Then,
$$|I(V,\L)|\lesssim n^{1/2}m^{3/4}+n^{2/3}m^{1/3}q_2^{1/3}+n+m.$$
\end{thm}

\begin{thm}[\cite{SS14}] 
\label{thm-incidencesinR4}
Let $V$ be a set of $n$ points and $\L$ be a set of $m$ lines in $\R^4$. Let $q_3$ be the maximum number of lines of $\L$ contained in a quadric hypersurface or a hyperplane and $q_2$ be the maximum number of lines of $\L$ contained in a 2-flat. Then,
$$|I(V,\L)|\lesssim 2^{c\sqrt{\log n}}\left(n^{2/5}m^{4/5}+n\right)+n^{1/2}m^{1/2}q_3^{1/4}+n^{2/3}m^{1/3}q_2^{1/3}+m$$ for some absolute constant $c$.
\end{thm}

\begin{proof}[Proof of Theorem~\ref{thm-conj-d=3}]
Let $\L=\L_r(V)$ be the set of $r$-rich lines and let $|\L|=m$. Let $q_2$ be the maximum number of lines of $\L$ contained in a hyperplane. By Theorem~\ref{thm-incidencesinR3}, $$rm\le |I(V,\L)|\lesssim n^{1/2}m^{3/4}+n^{2/3}m^{1/3}q_2^{1/3}+n+m.$$ From this it follows that
$$m \lesssim \frac{n^2}{r^4}+\frac{nq_2^{1/2}}{r^{3/2}}+\frac{n}{r}.$$
Now we will upper bound $q_2$. Let $\L'\subset \L$ be a set of $q_2$ lines contained in some hyperplane $H$ and let $V'=V\cap H$. We know that $|V'|\le s_2$. By applying Theorem~\ref{thm-r-richlines-2dimensions} to $\L',V'$ in $H$, we get $$q_2=|\L'|\lesssim \frac{|V'|^2}{r^3}+\frac{|V'|}{r} \le \frac{s_2^2}{r^3}+\frac{s_2}{r}\Rightarrow q_2^{1/2}\lesssim \frac{s_2}{r^{3/2}}+\frac{s_2^{1/2}}{r^{1/2}}.$$
Using this bound on $q_2$ we get, $$m\lesssim \frac{n^2}{r^4}+\frac{ns_2}{r^3}+\frac{ns_2^{1/2}}{r^2}+\frac{n}{r}\lesssim \frac{n^2}{r^4}+\frac{ns_2}{r^3}+\frac{n}{r}$$ where in the last step we used AM-GM inequality.
\end{proof}

\begin{proof}[Proof of Theorem~\ref{thm-conj-d=4}]
In this proof $c$ represents some absolute constant which can vary from step to step. Let $\L=\L_r(V)$ be the set of $r$-rich lines and let $|\L|=m$. Let $q_3$ be the maximum number of lines of $\L$ contained in a quadric hypersurface or a hyperplane and $q_2$ be the maximum number of lines of $\L$ contained in a 2-flat. By Theorem~\ref{thm-incidencesinR4}, $$rm\le |I(V,\L)|\lesssim 2^{c\sqrt{\log n}}\left(n^{2/5}m^{4/5}+n\right)+n^{1/2}m^{1/2}q_3^{1/4}+n^{2/3}m^{1/3}q_2^{1/3}+m$$. From this it follows that
\begin{equation}
\label{eq-thm-conj-d=4}
 m\lesssim 2^{c\sqrt{\log n}}\left(\frac{n^2}{r^5}+\frac{n}{r}\right)+\frac{nq_3^{1/2}}{r^2}+\frac{nq_2^{1/2}}{r^{3/2}}.
\end{equation}
By applying Theorem~\ref{thm-r-richlines-2dimensions} to the collection of $q_2$ lines contained in a 2-flat, we get $$q_2\lesssim\frac{s_2^2}{r^{3}}+\frac{s_2}{r}.$$
Now we will upper bound $q_3$. Let $\L'\subset \L$ be a set of $q_3$ lines contained in some quadric or hyperplane $Q$ and let $V'=V\cap Q$. We know that $|V'|\le s_3$. By applying Theorem~\ref{thm-incidencesinR4} again to $\L',V'$, we get $$rq_3\le |I(V',\L')| \lesssim 2^{c\sqrt{\log s_3}}\left(s_3^{2/5}q_3^{4/5}+s_3\right)+s_3^{1/2}q_3^{3/4}+s_3^{2/3}q_3^{1/3}q_2^{1/3}+q_3$$
$$\Rightarrow q_3\lesssim 	2^{c\sqrt{\log s_3}}\cdot \left(\frac{s_3^2}{r^5}+\frac{s_3}{r}\right)+\frac{s_3^2}{r^4}+\frac{s_3q_2^{1/2}}{r^{3/2}}.$$
Substituting these bounds on $q_3,q_2$ in Eq.~\ref{eq-thm-conj-d=4} and using AM-GM inequality a few times gives $$m\lesssim2^{c\sqrt{\log n}}\left(\frac{n^2}{r^5}+\frac{n}{r}\right)+\frac{ns_3}{r^4}\left(1+\frac{2^{c\sqrt{\log s_3}}}{r^{1/2}}\right)+\frac{ns_2}{r^3} \lesssim 2^{c\sqrt{\log n}}\left(\frac{n^2}{r^5}+\frac{ns_3}{r^4}+\frac{ns_2}{r^3}+\frac{n}{r}\right).$$
\end{proof}

\end{appendix}

\end{document}